\documentclass{article}
\usepackage{graphicx} 
\usepackage[left=1in, right=1in, top=1in, bottom=1in]{geometry}

\usepackage[utf8]{inputenc}
\usepackage{amsthm}
\usepackage{amssymb}
\usepackage{amsmath}
\usepackage{mathtools}
\usepackage{csquotes}
\usepackage{enumitem}   
\usepackage{tikz}
\usepackage{graphicx}
\usepackage{svg}

\usepackage[colorlinks=true,allcolors=blue,allbordercolors=white]{hyperref} 

\newcommand{\defn}[1]{\textit{#1}} 

\def\draft{1}
\newcommand{\todo}[1]{\ifnum\draft=1{\color{red}\textbf{TODO:} #1}\fi}

\usepackage[capitalise,nameinlink,noabbrev]{cleveref}
\crefname{equation}{}{} 
\crefname{enumi}{Step}{} 

\usepackage{thmtools}
\usepackage{thm-restate}

\declaretheorem[numberwithin=section]{theorem}

\declaretheorem[sibling=theorem]{definition}

\declaretheorem[sibling=theorem]{proposition}
\declaretheorem[sibling=theorem]{lemma}
\declaretheorem[sibling=theorem]{claim}
\declaretheorem[sibling=theorem]{corollary}
\declaretheorem[sibling=theorem]{example}
\declaretheorem[sibling=theorem]{conjecture}

\usepackage{mdframed}
\usepackage{framed}

\newcommand{\F}{\mathbb{F}}

\newcommand{\N}{\mathbb{N}}

\usepackage[
    backend=biber,
    style=alphabetic,
    maxnames= 10,
    maxalphanames=10,
]{biblatex}
\DeclareMathOperator{\ex}{ex}
\DeclareMathOperator{\PG}{PG}
\DeclareMathOperator{\Hea}{Hea}

\newcommand{\ind}{\text{-ind}}
\newcommand{\core}{\text{-core}}

\setlist[enumerate]{itemsep=\smallskipamount,parsep=0pt,label={\rm \roman*)}}
\setlist[itemize]{itemsep=\smallskipamount,parsep=0pt}

\title{New Bounds for Induced Tur\'an Problems}
\author{Nathan S. Sheffield}
\date{}

\begin{document}

\maketitle

\begin{abstract}
    In a recent paper, Hunter, Milojevi\'c, Sudakov and Tomon consider the maximum number of edges in an $n$-vertex graph containing no copy of the complete bipartite graph $K_{s,s}$ and no \emph{induced} copy of a ``pattern'' graph $H$. 
    They conjecture that, for $s \geq |V(H)|$, this ``induced extremal number'' differs by at most a constant factor from the standard extremal number of $H$.
    Towards this, we give bounds on the induced extremal number in terms of degeneracy, which establish some non-trivial relationship between the induced and standard extremal numbers in general. 
    We also show that (as in the case of standard extremal numbers) the induced extremal number is dominated by that of the 2-core of a single connected component.
    Finally, we present some graphs arising from incidence geometry which may serve as counterexamples to the conjecture.

\end{abstract}

\section{Introduction}

The extremal number of a graph $H$, denoted $\ex(n, H)$, is the maximum number of edges in an $n$-vertex graph with no copy of $H$ as a subgraph (we will sometimes refer to $H$ as the ``pattern'' graph to be avoided). It is known by a result of Erd\H{o}s, Stone and Simonovits that $\ex(n, H)= n^2 \left(1 - \frac{1}{\chi(H)-1}\right) + o(n^2)$ for all $H$, where $\chi(H)$ is the chromatic number~\cite{ess-2, ess-1}. However, for bipartite graphs, this only allows us to  $\ex(n, H) \leq o(n^2)$; understanding the extremal exponents of bipartite graphs is a major ongoing area of research.\\

What if, instead of just finding $H$ as a subgraph, we are interested in finding an \emph{induced} copy of $H$? We could define $\ex(n, H\ind)$ to be the maximum number of edges in an $n$-vertex graph with no induced copy of $H$. However, this notion is rather uninteresting: unless $H$ is itself complete, the complete graph $K_n$ avoids it as an induced subgraph, so $\ex(n, H\ind) = \binom{n}{2}$. One might ask, though, whether this is in some sense the \emph{only} way to avoid induced copies of $H$ without avoiding $H$ altogether. That is, perhaps if a graph has enough edges to force many copies of $H$, but contains no induced copy of $H$, then there must exist some portion of the graph with very high density. Hunter, Milojevi\'c, Sudakov and Tomon propose the following conjecture:

\begin{conjecture}[\cite{main-paper}]\label{conj:connected}
    For any $s \in \N$, and any connected bipartite $H$, 
    \[\ex(n, \{K_{s,s}, H\ind\}) \leq O_{s, H}(\ex(n,H)).\]
\end{conjecture}

We always have $\ex(n, \{K_{s,s}, H\ind\}) \geq \ex(n, \{K_{s,s}, H\})$, and if $s \geq |V(H)|$ then it is clear that we have $\ex(n, \{K_{s,s}, H\}) = \ex(n, H)$. So, this conjecture is equivalent to the claim that, for sufficiently large $s$, the extremal numbers $\ex(n, \{K_{s,s}, H\ind\})$ and $\ex(n,H)$ can differ by at most a constant factor (where the constant depends on $s$ and $H$ --- we will treat these as constants in our asymptotic notation for the rest of the paper). As a heuristic for why such a conjecture might be reasonable, one could consider searching for copies of $H$ in the Erd\H{o}s--R\'enyi random graph $G(n,p)$: if $p \geq \Omega(1)$, then there will be $\Theta(n^{2s})$ copies of $K_{s,s}$, but if $p \leq o(1)$, then all but a subconstant fraction of the copies of $H$ will be induced copies.\\

Hunter, Milojevi\'c, Sudakov and Tomon provide some evidence for their conjecture by directly reproducing many of the best known asymptotic upper bounds on standard extremal numbers. In particular, they show $\ex(n, \{K_{s,s}, C_{2k}\ind \}) \leq O(n^{1+1/k})$ for $C_{2k}$ a length-$2k$ cycle, $\ex(n, \{K_{s,s}, Q_8 \ind \}) \leq O(n^{8/5})$ for $Q_8$ the skeleton of a 3-dimensional cube, and $\ex(n, \{K_{s,s}, H\ind \}) \leq O(n^{2 - 1/r})$ for any bipartite $H$ with maximum degree $r$ on one side ---  all of which match the corresponding bounds known for $\ex(n, H)$~\cite{main-paper}. Fox, Nenadov and Pham show that this last result can be extended to give $\ex(n, \{K_{s,s}, H\ind \}) \leq O(n^{2 - 1/r})$ whenever one side of $H$ has at most $r$ vertices that are complete to the other side, and all other vertices on that side have degree at most $r$~\cite{fox-induced}. 
Under the same restriction on $H$, plus the additional requirement that $H$ contain no copy of $K_{r,r}$, Axenovich and Zimmerman show that any bipartite graph avoiding $H$ and $K_{s,s}$ can have only $o(n^{2- 1/r})$ edges~\cite{axenovich-induced} --- this provides an induced analogue of a result of Sudakov and Tomon~\cite{sudakov-tomon-noktt}.
It is also known by a result of Scott, Seymour and Spirkl that $\ex(n, \{K_{s,s}, T\ind\}) \leq O(n)$ when $T$ is a tree~\cite{induced-trees}; Hunter, Milojevi\'c, Sudakov and Tomon show that the constant factor in that result can be made polynomial in $s$~\cite{main-paper}. \\

\subsection{Main results}

Hunter, Milojevi\'c, Sudakov and Tomon's original statement of \cref{conj:connected} originally excluded the word ``connected'' --- however, we note in \cref{sec:ez-counterex} that this original statement has a simple counterexample. The graph $H$ we give as a counterexample consists of three vertices, with two joined by a single edge. To rule out such trivial counterexamples, one could choose either to require $H$ to be connected, or simply require that $H$ have at least two edges. Our first result implies that both modifications to the conjecture are equivalent.

\begin{restatable*}{proposition}{connected}\label{prop:wlog-connected}
    For any $s\in\N$, if $H$ is the disjoint union of two subgraphs $H_1$ and $H_2$, then 
    \[\ex(n, \{K_{s,s}, H\ind\}) \leq O(\ex(n, \{K_{s,s}, H_1\ind \}) + \ex(n, \{K_{s,s}, H_2\ind \}) + n).\]
\end{restatable*}

One other consequence of \cref{prop:wlog-connected} is that, as long as $\ex(n, \{K_{s,s}, H\ind\}) \geq Cn$ for some sufficiently large constant $C$, this value remains unchanged up to constant factors when we remove all isolated vertices from $H$. For the standard extremal number, it is easy to see that such a statement holds even if we also remove all vertices of degree $1$. In \cref{sec:2core}, we reproduce this fact in the induced setting.

\begin{definition}
    For $k \geq 1$, let $k\core(G)$ be the largest induced subgraph of $G$ with minimum degree at least $k$.
\end{definition}

\begin{restatable*}{theorem}{twocore}\label{thm:wlog-2core}
    For any $s\in\N$ and any graph $H$, we have 
    \[\ex(n, \{K_{s,s}, H\ind\}) \leq O(\ex(n, \{K_{s,s}, 2\core(H)\ind\}) + n).\]
\end{restatable*}

Together with the results of \cite{main-paper} on cycles, \cref{prop:wlog-connected} and \cref{thm:wlog-2core} imply that we have $\ex(n, \{K_{s,s}, H\ind\}) \leq O(n^{1+1/k})$ for any bipartite $H$ with girth at least $2k$, as long as every connected component of $H$ contains at most a single cycle.\\

In \cref{sec:degeneracy}, with the goal of determining a relationship between $\ex(n,H)$ and $\ex(n,\{K_{s,s}, H\ind\})$ in general, we give an upper bound in terms of the degeneracy of $H$.

\begin{definition}
    For $k \geq 1$, a graph $G$ is called $k$-degenerate if $(k+1)\core(G)$ is empty. We define the degeneracy of $G$ to be the minimum $k$ such that $G$ is $k$-degenerate.
\end{definition}

Although there is no general method known for determining extremal numbers of bipartite graphs, it is known that the degeneracy of the pattern graph always offers a somewhat reasonable proxy: for any bipartite $H$ of degeneracy $r$, we have $\Omega(n^{2 - 2/r}) \leq \ex(n, H) \leq O(n^{2 - 1/4r})$~\cite{aks-degenerate}, and it is conjectured that the upper bound can be strengthened to $O(n^{2 - 1/r})$ \cite{erdos-favorites}. We show that degeneracy also controls the induced extremal number, although we achieve weaker dependency on $r$ than what is known for standard extremal numbers:

\begin{restatable*}{theorem}{degeneracy}\label{thm:degeneracy-controls-ind}
    For any $s\in\N$, and any bipartite $H$ of degeneracy $r$, we have 
    \[\ex(n, \{K_{s,s}, H\ind\}) \leq O(n^{2 - 1/(20r^4)}).\]
\end{restatable*}

The fact that degeneracy controls both the standard and induced extremal numbers immediately establishes some nontrivial relationship between them: if $\ex(n,H) = O(n^\alpha)$ and $\ex(n, \{K_{s,s}, H\ind\}) = \Omega(n^\beta)$, then we can show unconditionally that $\beta \leq 2 - \frac{(2 - \alpha)^4}{320}$ (whereas \cref{conj:connected} would hold that $\beta \leq \alpha$). It would be interesting to quantitatively improve this relationship by showing an upper bound of the form $\ex(n, \{K_{s,s}, H\ind\}) \leq O(n^{2-1/cr})$ for some constant $c$ independent of $r$; we make some progress towards this goal by finding copies of $H$ which, while not necessary fully induced, avoid some particular subset of $H$'s non-edges. \\

Finally, in \cref{sec:real-counterexs}, we discuss the possibility that \cref{conj:connected} could be false. We note several graphs $H$ where $\ex(n, \{K_{2,2}, H\ind) = \Theta(n^{3/2})$, but such that we may plausibly have $\ex(n, H) \leq o(n^{3/2})$. However, due to the difficulty of determining standard extremal numbers, we are unable to prove such upper bounds. 

\subsection{Additional related work}

The structure of graphs avoiding given induced subgraphs has been an active area of research for some time. Much existing work is motivated by the Erd\H{o}s--Hajnal Conjecture, which claims that, for any $H$, any $n$-vertex graph with no induced copy of $H$ must contain a clique or independent set of size polynomial in $n$~\cite{erdos-hajnal-conj, erdos-hajnal-survey}. Towards this conjecture, Fox and Sudakov have shown that any graph avoiding induced copies of $H$ must contain either a complete bipartite graph or independent set of polynomial size~\cite{fox-erdos-hajnal}. \\

To our knowledge, Hunter, Milojevi\'c, Sudakov and Tomon are the first to systematically consider $\ex(n, \{K_{s,s}, H\ind\})$ for general bipartite $H$; however, the problem of forbidding one induced graph and one non-induced graph has recieved prior attention. K\"uhn and Osthus have shown that $\ex(n, \{K_{s,s}, \mathcal{H}\ind\}) \leq O(n)$, where $\mathcal{H}$ is the family of all subdivisions of a given graph $H$~\cite{kuhn2004induced}. Loh, Tait, Timmons and Zhou showed that $\ex(n, \{K_r, K_{s,t}\ind\}) \leq O(n^{2 - 1/s})$ for all $r,s,t$ ~\cite{loh-induced}, prompting further consideration of $\ex(n, \{F, H\ind\})$ for non-bipartite $F$~\cite{ergemlidze-turan, illingworth-note, illingworth-k2t}.\\

While our primary concern in this paper is the dependence on $n$, some existing work also asks about the dependence of $\ex(n, \{K_{s,s}, H\ind\})$ on $s$. The work of Fox, Nenadov and Pham~\cite{fox-induced} deals with a generalization of this question: as opposed to forbidding $K_{s,s}$, they consider host graphs which are required to be $(c, s)$-sparse, meaning that every pair of vertex subsets $A, B \subseteq V(G)$ with $|A|, |B| \geq s$ have $e(A, B) \leq (1-c)|A||B|$. Note that for $c \leq 1/s^2$ this condition is equivalent to forbidding $K_{s,s}$, however for larger $c$ this is a more restrictive condition --- Fox, Nenadov and Pham develop techniques to attain improved dependence on $s$ for large $c$, and demonstrate that in some cases these methods also interpolate to a good dependence on $s$ when $c$ is small.

\section{Reducing to a single connected component}\label{sec:ez-counterex}

We begin by noting our counterexample to the original statement of \cref{conj:connected}, which did not require $H$ to be connected.
\begin{example}
        Let $H$ be the graph consisting of two vertices connected by an edge, and a third isolated vertex. Note that 
    \[\ex(n, H) = \begin{cases}
        1 \text{ if $n = 2$}\\
        0 \text{ otherwise,}
    \end{cases}\]
    since a subgraph isomorphic to $H$ just consists of a single edge plus some third vertex. However, in the star graph $K_{1, n-1}$, any three vertices induce either an independent set or a path, depending on whether one of the three vertices is the center of the star. Since $K_{1,n-1}$ contains no cycles, we have exhibited a graph on $n-1$ edges with no copy of $K_{2,2}$ and no induced copy of $H$, giving $\ex(n, \{K_{2,2}, H\ind\}) \geq n-1 \geq \omega(\ex(n,H))$. 
\end{example}

This appears more like an annoying edge case than a serious issue with the spirit of the conjecture, so it seems appropriate to exclude it by requiring $H$ to be connected. But an alternative way to modify the conjecture would be to simply exclude this one particular counterexample --- that is, to allow disconnected $H$, but specially disregard the case where $H$ consists only of a single edge and some number of isolated vertices.



\begin{conjecture}\label{conj:two-edges}
    For any $s \in \N$, and any bipartite $H$ \emph{with at least two edges}, 
    \[\ex(n, \{K_{s,s}, H\ind\}) \leq O(\ex(n,H)).\]
\end{conjecture}

Since the only connected graph with a single edge is $K_{1,1}$, it is clear that \cref{conj:two-edges} implies \cref{conj:connected}. Our next observation will show that the reverse implication also holds, so that these two versions of the conjecture are equivalent.

\connected

\begin{proof}
    The K\H{o}v\'ari--S\'os--Tur\'an Theorem gives that $\ex(n, K_{s,s}) \leq O(n^{2-1/s})$, which means in particular that any sufficiently large graph with constant edge density must contain a copy of $K_{s,s}$. Let $N$ be such that, for any $N' \geq N$, any $N'$-vertex graph with at least $\frac{N'^2}{|V(H))|^2}$ edges must contain $K_{s,s}$ as a subgraph. We will show that, for any $n$, any $n$-vertex graph $G$ with at least $\ex(n, \{K_{s,s}, H_1\ind\}) + \ex(n, \{K_{s,s}, H_2\ind\}) + N |V(H)| n$ edges must contain either an induced copy of $H = H_1 \sqcup H_2$, or a copy of $K_{s,s}$.\\

    The proof follows from a simple supersaturation argument. We claim that, if $G$ contains no $K_{s,s}$, then it must contain $N$ induced copies of $H_1$ and $N$ induced copies of $H_2$, all of which are vertex-disjoint. This can be shown greedily: suppose we have already found fewer than disjoint $N$ copies of $H_1$ and $N$ copies of $H_2$, and consider the subgraph of $G$ induced by all vertices not included in any of those copies. That subgraph contains at least $|E(G)| - N |V(H_1)| n - N |V(H_2)|n = |E(G)| - N|V(H)|n \geq \ex(n, \{K_{s,s}, H_1\ind\}) + \ex(n, \{K_{s,s}, H_2\ind\})$ edges, and thus we can find another disjoint induced copy of whichever of $H_1$ or $H_2$ we desire. \\

    Now, suppose for contradiction that $G$ contains no induced copy of $H$. This means that each of those copies of $H_1$ must have an edge to each of those copies of $H_2$, since otherwise the pair would induce a copy of $H$. Thus, the subgraph induced by all $N$ disjoint copies of $H_1$ and $H_2$ together --- which is a graph on $N |V(H)|$ vertices --- contains at least $N^2$ edges. By our choice of $N$, this guarantees that the graph contains a copy of $K_{s,s}$.
\end{proof}

\begin{corollary}
    \cref{conj:connected} $\iff$ \cref{conj:two-edges}.
\end{corollary}
\begin{proof}
    The \cref{conj:two-edges} $\implies$ \cref{conj:connected} direction is trivial, as 
    the only connected graph with a single edge is $K_{1,1}$, in which case we have $\ex(n,K_{1,1}) = \ex(n, \{K_{s,s}, K_{1,1}\ind\}) = 0$. To show the reverse direction, we assume \cref{conj:connected}, and demonstrate \cref{conj:two-edges} by contradiction. Fix $s \in \N$, and let $H$ be a bipartite graph with at least two edges such that $\ex(n, \{K_{s,s}, H\ind\}) \geq \omega(\ex(n, H))$ --- in fact, take $H$ to have the minimum possible number of vertices among all such graphs. $H$ cannot be connected, since otherwise \cref{conj:connected} would apply. So, we can split $H$ into two disconnected components $H = H_1 \sqcup H_2$. By \cref{prop:wlog-connected}, we have $\ex(n, \{K_{s,s}, H_1\ind\}) + \ex(n, \{K_{s,s}, H_2\ind\}) + O(n) \geq \omega(\ex(n, H))$. Since $H$ has at least two edges, $\ex(n,H) \geq \Omega(n)$, so either $\ex(n, \{K_{s,s}, H_1\ind\}) \geq \omega(\ex(n, H)) \geq \omega(\ex(n, H_1))$ or $\ex(n, \{K_{s,s}, H_2\ind\}) \geq \omega(\ex(n, H)) \geq \omega(\ex(n, H_2))$ --- but either of these would contradict minimality of $H$.
\end{proof}

So, even if we care about the disconnected case, it suffices to consider the question only for connected $H$. In the next section, we will show that it also suffices to consider the question only for $H$ of minimum degree at least two.

\section{Reducing to the 2-core}\label{sec:2core}

It is known that \cref{conj:connected} holds when $H$ is restricted to be a tree --- i.e., $\ex(n, \{K_{s,s}, H\ind\}) \leq O(\ex(n,H)) \leq O(n)$ whenever $H$ is a tree \cite{induced-trees, main-paper}. For the standard extremal number, it is easy to show that degree-one vertices simply do not affect the extremal number: not only is $\ex(n, T) \leq O(n)$ for all trees $T$, but more generally attaching a tree to a single vertex of any graph $H$ (not necessarily a tree) can change $H$'s extremal number by at most an additive $O(n)$. In this section, we will show how to recover a statement of this form in the induced setting, too. \\

\twocore

The first step in the proof is a standard regularization argument, originally due to Erd\H{o}s and Simonovits.

\begin{lemma}[\cite{ess-1}; see also \cite{furedi-history}]\label{lem:regularize}
    For any $\alpha \in [0,1]$, and any $N$-vertex graph $G$ with at least $N^{1+\alpha}$ edges, there exists an induced subgraph $G' \subseteq G$ with $n \geq N^{\frac{\alpha - \alpha^2}{1 + \alpha}}$ vertices and at least $\frac{2}{5}n^{1+\alpha}$ edges, such that the maximum degree in $G'$ is at most a $20 \cdot 2^{1/\alpha^2}$ factor larger than the minimum degree in $G'$.
\end{lemma}

The statement in \cite{ess-1} does not mention that $G'$ can be taken to be induced, but as noted in \cite{axenovich-induced} this is immediate from the proof. \\

\cref{lem:regularize} will allow us to show the key technical tool in our proof of \cref{thm:wlog-2core}: a supersaturation result allowing us to find many induced copies of a subgraph that all share only one specified vertex.

\begin{lemma}\label{lem:flower}
    Fix any $s,t \in \N$, any graph $H$ containing at least one cycle, and any $v \in V(H)$. There exists some constant $C$ independent of $N$ such that for any $N$-vertex $K_{s,s}$-free graph $G$ on $C \cdot \ex(N, \{K_{s,s}, H\ind\})$ edges, there exist $t$ induced copies of $H$ in $G$ such that any pair of copies overlap on exactly one vertex, and that vertex is the image of $v$ for both copies.
\end{lemma}
\begin{proof}
    Note that any graph containing a cycle must have extremal number at least $N^{1+\alpha}$ for some constant $\alpha$, since $\ex(N, C_{2k}) \geq N^{1+1/2k}$ (this follows from considering a random host graph)~\cite{furedi-history}. So, \cref{lem:regularize} implies that, for any $C'$, if $C$ is chosen sufficiently large and $G$ is an $N$-vertex graph on $C\cdot \ex(N, \{K_{s,s}, H\ind\})$ vertices, then we can pass to a subgraph $G' \subseteq G$ with $n$ vertices and at least $C' \cdot (\ex(n, \{K_{s,s}, H\ind\}))$ edges, such that $G$ has maximum degree at most $\frac{\Delta \cdot E(G')}{n}$ and minimum degree at least $\frac{E(G')}{\Delta n}$, for a constant $\Delta = 20 \cdot 2^{1/\alpha^2}$ independent of $C'$. Now, for some $k$ to be determined later, consider the following randomized procedure:

    \begin{enumerate}
        \item Choose a uniform random $\frac{n}{k}$ vertices $R \subseteq V(G')$.
        \item Choose a uniform random vertex $u \in R$.
        \item Declare the procedure to have succeeded if there exists an induced embedding $\pi$ of $H$ in $R$ such that $\pi(v) = u$. 
    \end{enumerate}

    We claim that as long as $N$ (and therefore $n$) is sufficiently large, this procedure has a reasonably high success probability. The first necessary observation is that, with high probability, the graph remains nearly regular upon subsampling to $R$.

    \begin{claim}\label{clm:stay-reg}
        With probability at least $1-2^{-\mathsf{poly}(n)}$, every vertex in $R$ has at least $\frac{|E(G')|}{2 k \Delta n}$ and at most $\frac{2 \Delta \cdot |E(G')|}{kn}$ neighbors in $R$.
    \end{claim}
    \begin{proof}
        For any specific vertex $x \in V(G')$, consider the probability that $x$ has fewer than $\frac{|E(G')|}{2 k \Delta n}$ neighbors in $R$. We know that $x$ has degree at least $\frac{|E(G')|}{\Delta n}$, so this probability is at most the chance that fewer than a $\frac{1}{2k}$ fraction of $x$'s neighbors are chosen to belong to $R$. Since $R$ is a uniform random $1/k$-fraction of all vertices, a Chernoff bound guarantees that this probability is exponentially small in the size of $x$'s neighborhood~\cite{pm}, which is at least $\frac{|E(G')|}{\Delta n} \geq \frac{n^{\alpha}}{\Delta} = \mathsf{poly}(n)$. Union bounding over all $x \in V(G')$, this means that the probability that any of them have fewer than $\frac{|E(G')|}{2 k \Delta n}$ neighbors in $R$ is exponentially small (i.e., $2^{-\mathsf{poly}(n)}$). The upper bound on neighborhood size is identical.
    \end{proof}

    We then observe that, so long as $R$ remains nearly regular, we can find induced copies of $H$ within $R$ making use of a substantial fraction of the vertices. 

    \begin{claim}\label{clm:stay-reg=good}
        If the subgraph induced by $R$ has $|E(R)| \geq 2 \cdot \ex(n, \{K_{s,s}, H\ind\})$, and maximum degree at most a $4 \Delta^2$ factor larger than minimum degree, then for at least $\frac{|V(R)|}{8\Delta^2}$ vertices $u \in R$, there exists an induced embedding $\pi$ of $H$ in $R$ such that $\pi(v) = u$.
    \end{claim}
    \begin{proof}
        Suppose there are fewer than $\frac{|V(R)|}{8\Delta^2}$ vertices of $R$ that serve as the image of $v$ under some embedding of $H$. Since the graph induced by $R$ has maximum degree at most $\frac{4\Delta^2 \cdot |E(R)|}{|V(R)|}$, the subgraph induced by all vertices of $R$ other than those contains at least $|E(R)| - \left(\frac{|V(R)|}{8\Delta^2}\right)\left( \frac{4\Delta^2 \cdot |E(R)|}{|V(R)|}\right) = \frac{|E(R)|}{2} \geq \ex(n, \{K_{s,s}, H\ind\})$ edges. So, it must contain an induced copy of $H$, which must have some vertex corresponding to $v$, which is a contradiction since we have removed all possible images.
    \end{proof}

    As long as $C' \ge 2k\Delta$, \cref{clm:stay-reg} and \cref{clm:stay-reg=good} together guarantee that our random process has success probability at least $\frac{1}{10 \Delta^2}$ for sufficiently large $n$, since once we condition on an event with probability going to $1$ in $n$, a random vertex in $R$ has at least a $\frac{|V(R)|}{8\Delta^2}$ chance of leading to success.\\
    
    But now, note that we could alternatively have performed the procedure in the following order:
    
    \begin{enumerate}
        \item Choose a uniform random vertex $u \in V(G')$.
        \item Choose a uniform random partition of all other vertices of $G$ into $k$ equal-sized color classes.
        \item Choose one of the colors uniformly at random to call $R$.
        \item Declare the procedure to have succeeded if there exists an induced embedding $\pi$ of $H$ in $R \cup \{u\}$ such that $\pi(v) = u$.
    \end{enumerate}

    Up to an additive difference of $1$ in the number of vertices chosen in $R$, which affects the distribution negligibly, this process gives the same distribution over $R$ and $u$ as the one originally specified, and so has the same success probability of at least $\frac{1}{10 \Delta^2}$. By averaging, there exists some way to perform steps i and ii such that the process still has success probability at least $\frac{1}{10 \Delta^2}$ over step iii. This means that, for some vertex $u \in V(G')$ and some partition into $k$ colors, at least $\frac{k}{10\Delta^2}$ of the color classes contain an induced copy of $H$ mapping $v$ to $u$. Since each of these copies are (aside from $u$) of different vertex colors and hence disjoint, choosing $k = 10 \Delta^2 t$ gives the desired statement.
\end{proof}

\begin{proof}[Proof of \cref{thm:wlog-2core}]
    The statement is already known for trees~\cite{induced-trees, main-paper}, and so \cref{prop:wlog-connected} immediately implies it for forests. \cref{prop:wlog-connected} also allows us to remove any isolated vertices without affecting the induced extremal number by more than an additive $O(n)$. Thus, we need only consider the case where $H$ is a connected graph containing a cycle. Proceeding by induction on the number of edges of $H$, it suffices to show for all such $H$ that $\ex(n, \{K_{s,s}, H^+\ind) \leq O(\ex(n, \{K_{s,s}, H\ind))$, where $H^+$ is obtained from $H$ by adding a single vertex $u$, and a single edge $(u,v)$ to some $v \in V(H)$. That is, we will show that adding a single degree-1 vertex to $H$ cannot change its induced extremal number by more than a constant multiplicative factor. \\
    
    By \cref{lem:flower}, there exists some $C$ such that any $n$-vertex graph with at least $C \cdot \ex(n, \{K_{s,s}, H\ind\})$ edges must contain $s$ copies of $H$, any pair of which overlap exactly on the vertex corresponding to $v$. Let $G$ be a graph with at least $2C\cdot \ex(n, \{K_{s,s}, H\ind\}) + 2\left(s(|V(H)|-1)^s + s(|V(H)|-1)\right)n$ edges. By repeatedly removing vertices with degree less than half the average, we can find an induced subgraph $G' \subseteq G$ with minimum degree at least $\frac{C}{n}\ex(n, \{K_{s,s}, H\ind\}) + s(|V(H)|-1)^s + s(|V(H)|-1)$. Now, by \cref{lem:flower}, we can find induced copies $H_1, \dots, H_s$ of $H$ in $G'$ that all overlap only on $\pi(v)$, where $\pi(v)$ is the image of $v$ in all copies. By our bound on minimum degree, we know that $\pi(v)$ has degree at least $s(|V(H)|-1)^s + (|V(H)|-1)$, and hence has at least $s(|V(H)|-1)^s$ neighbors $u_{1}, \dots, u_{\left( s(|V(H)|-1)^s \right)}$ not contained in any of our identified $H_i$. \\
    
    If $G'$ contains no induced copy of $H^+$, then for every $i, j$, there must be an edge between $u_i$ and some vertex of $H_j \setminus \{\pi(v)\}$; choose an arbitrary such edge for each $i,j$. There are only $(|V(H)|-1)^s$ ways to choose one element of $H_j \setminus \{\pi(v)\}$ for each $j$, so by the pigeonhole principle there must exist a set of $s$ many $u_i$'s for which we have chosen the exact same $s$-tuple of neighbors. These vertices and that $s$-tuple of neighbors form a copy of $K_{s,s}$.
\end{proof}

\begin{corollary}
    If $H$ is any bipartite graph of girth at least $2k$ such that each connected component contains at most a single cycle, then for any $s$ we have $\ex(n, \{K_{s,s}, H\}) \leq O(n^{1 + 1/k})$.
\end{corollary}
\begin{proof}
    Hunter, Milojevi\'c, Sudakov and Tomon have shown that that $\ex(n, \{K_{s,s}, C_{2\ell}) \leq O(n^{1+1/\ell})$ for all $\ell$~\cite{main-paper}. By \cref{thm:wlog-2core}, this gives that $\ex(n, \{K_{s,s}, F) \leq O(n^{1+1/k})$ for any $F$ whose $2$-core is $C_{2\ell}$, we have $\ell \geq k$ -- in other words, any bipartite graph of girth at least $2k$ containing only a single cycle. Then, by \cref{prop:wlog-connected}, the disjoint union $H$ of any constant number of such graphs must also have $\ex(n, \{K_{s,s}, H\}) \leq O(n^{1 + 1/k})$.
\end{proof}

\section{Control by degeneracy}\label{sec:degeneracy}

The results of \cref{sec:2core} give us a new class of graphs where we have upper bounds on induced extremal numbers matching those for standard extremal numbers. However, we still have no relationship between induced and standard extremal numbers in general. Short of a proof of \cref{conj:connected}, it would be useful to at least rule out that $\ex(n, \{K_{s,s}, H\ind\})$ can be arbitrarily large in terms of $\ex(n,H)$. That is, we would like to show that for every $\varepsilon$ there exists a $\delta$ such that, for all $H$, if $\ex(n,H) \leq n^{2 - \varepsilon}$, then $\ex(n, \{K_{s,s}, H\ind\}) \leq n^{2-\delta}$. In this section, we will obtain such a result.\\

The relevant fact is that, while we know no general way of computing extremal numbers, the degeneracy of $H$ offers a reasonable approximation, giving both lower and upper bounds. Specifically, we know that, for some constants $c > k$, for any $r$ and any bipartite $H$ of degeneracy $r$, we have $\Omega(n^{1 - 1/kr}) \leq \ex(n, H) \leq O(n^{1 - 1/cr})$. (The current best known values of $c$ and $k$ are $1/2$ and $4$, respectively, although it is conjectured that these can be improved~\cite{aks-degenerate, erdos-favorites}.) An upper bound on \emph{induced} extremal numbers in terms of degeneracy would therefore allow us to constrain the induced extremal number of a graph in terms only of its non-induced extremal number. In \cref{sec:weak-sauce}, we will show a bound of the form $\ex(n, \{K_{s,s}, H\ind\}) \leq O(n^{2 - 1/cr^c})$, which will establish some such relationship. To strengthen that relationship, it would be interesting to show a bound of the form $\ex(n, \{K_{s,s}, H\ind\}) \leq O(n^{2 - 1/cr})$; in \cref{sec:strong-sauce} we discuss partial progress towards such a stronger quantitative bound.

\subsection{Bounding induced extremal numbers in terms of degeneracy }\label{sec:weak-sauce}

As in the argument for standard extremal numbers, our upper bound relies on the technique of dependent random choice. In this section, we will be able to use directly the following standard result, whereas in the next section we will have to unfold its proof to obtain some stronger guarantees:

\begin{lemma}[\cite{aks-degenerate}, \cite{fox2011dependent}]\label{lem:vanilla-drc}
    For any $r, t \geq 2$, and any $n$-vertex graph $G$ with at least $n^{2-1/(t^3r)}$ edges, there exist nonempty subsets $U_1, U_2 \subseteq V(G)$ such that every $r$-tuple of (not necessarily distinct) vertices in $U_1$ has at least $n^{1 - 1.8/t}$ common neighbors in $U_2$, and likewise every $r$-tuple of vertices in $U_2$ has at least $n^{1-1.8/t}$ common neighbors in $U_1$.
\end{lemma}

We will also make use of the fact that, in a $K_{s,s}$-free graph, at most a constant number of vertices are neighbors with a constant fraction of any sufficiently large vertex set.

\begin{lemma}\label{lem:electrocute}
    For any $\varepsilon > 0$, any $K_{s,s}$-free graph $G$, and any vertex subset $S \subseteq V(G)$ with $|S| \geq \frac{2s}{\varepsilon}$, there exist at most $\left(\frac{2s}{\varepsilon}\right)^s$ vertices $v \in V(G)$ such that $|N(v) \cap S| \geq \varepsilon |S|$.
\end{lemma}
\begin{proof}
    Suppose there exist $x = \left(\frac{2s}{\varepsilon}\right)^s$ vertices $\{v_1,\dots,v_x\} \subseteq V(G)$ such that $|N(v_i) \cap S| \geq \varepsilon |S|$. Taking $S$ on one side, and these vertices $\{v_1, \dots, v_x\} \setminus S$ on the other, we can define a bipartite subgraph with part sizes at most $|S|$ and $x$, and at least $x \cdot \varepsilon |S| - \binom{|S|}{2} \geq 2s x^{1-1/s} |S| - \left(\frac{2s}{\varepsilon}\right)^2 \geq s^{1/s} x^{1-1/s} |S| + sx$ edges. 
    The asymmetric version of the K\H{o}v\'ari--S\'os--Tur\'an theorem guarantees that any bipartite graph with part sizes $\ell$ and $r$, and at least $ s^{1/s}r^{1-1/s}\ell + sr$ edges, must contain a copy of $K_{s,s}$~\cite{kst, zarankiewicz-2}.
    So, we have found a $K_{s,s}$ in $G$, which is contradiction.    
\end{proof}

In order to find an induced embedding, we will first apply \cref{lem:vanilla-drc}, then use \cref{lem:electrocute} to show that an appropriately-chosen random embedding of our pattern graph $H$ in the resulting pair of subsets will be an induced copy with high probability.

\degeneracy
\begin{proof}
    Let $H$ be an $r$-degenerate graph, and let $G$ be a $K_{s,s}$-free $n$-vertex graph with at least $n^{2 - 1/(20r^4)}$ edges. Applying \cref{lem:vanilla-drc} with $t = 2.71 r$, we obtain two vertex subsets $U_1, U_2 \subseteq V(G)$ such that every $r$-tuple in one subset has at least $n^{1 - 1/1.51r}$ common neighbors in the other.\\

    Let $v_1, \dots, v_{|V(H)|}$ be an ordering of the vertices of $H$ such that, for all $i$, vertex $v_i$ has at most $r$ neighbors $v_j$ with $j < i$. Such an ordering is guaranteed to exist by the assumption that $H$ is $r$-degenerate. We will consider embedding the vertices of $H$ one-at-a-time in that order, examining a large number of possible ways to do so. Specifically, for any tuple of numbers $w = (w_1, \dots, w_{|V(H)|}) \in [n^{1 - 1/1.51r}]^{|V(H)|}$, we will define an associated embedding $\pi_w: V(H) \to V(G)$. \\
    
    Fix an arbitrary ordering of the vertices in $U_1$ and $U_2$, respectively. To define the embedding of $v_i$, we will take the $w_i$th ``available'' option in this ordering, ensuring edges exist to $v_i$'s already-embedded neighbors. That is, for each $v_i$ in the left (resp.\ right) part of $H$, let $\pi_w(v_i)$ be the $w_i$th vertex of the common neighborhood $\bigcap_{j < i\colon (v_i, v_j) \in E(H)} N(\pi_w(v_j))$ in the ordering of $U_1$ (resp.\ $U_2$). Since each $v_i$ has at most $r$ earlier neighbors in the embedding, and any $r$ vertices in one $U_i$ have at least $n^{1-1/1.51r}$ common neighbors, the $w_i$th vertex defined thus will always exist. The image of each $\pi_w$ is a homomorphic copy of $H$: if $(v_i, v_j) \in E(H)$, then $(\pi_w(v_i),\pi_w(v_j)) \in E(G)$, since whichever of $v_i$ and $v_j$ is later in the ordering will be chosen from the neighborhood of the other. We will show that when $w$ is chosen uniformly at random among all elements of $[ n^{1 - 1/1.51r}]^{|V(H)|}$, the associated homomorphic copy of $H$ has nonzero probability of being an induced subgraph. \\

    To do so, we will define a notion of neighborhoods having large overlap, and claim that this is unlikely to occur. For any $r$-tuple $u_1, \dots, u_r \in U_j$, we define the set $A = A(u_1, \dots, u_r)$ to be the first $n^{1-1/1.51r}$ vertices of $\bigcap_i N(u_i)$ in the ordering of the other part of $U_{2-j}$. Note that our dependent random choice guarantees $|\bigcap_i N(u_i)| \geq n^{1-1/1.51r}$, so $A$ is well-defined.
    Now, for any other vertex $v \not\in \{u_1, \dots, u_r\}$, we say that $v$ \defn{electrocutes} $u_1, \dots, u_r$ if $v$ is adjacent to a large fraction of $A$. That is, $v$ electrocutes $u_1, \dots, u_r$ if we have $|N(v) \cap A| \geq \frac{1}{100 |V(H)|^2} \cdot n^{1-1/1.51r}$. Call a set $T \subseteq V(G)$ \defn{slippery} if some $r$-tuple $(u_1, \dots, u_r) \in T^r$ of the vertices is electrocuted by another vertex $v \in T$, $v \not\in \{u_1, \dots, u_r\}$.
        
    \begin{claim}\label{clm:not-slippery}
        If $n$ is sufficiently large, and $w$ is chosen uniformly from $[n^{1 - 1/1.51r}]^{|V(H)|}$, the image $\text{Im}_{\pi_w}(V(H))$ is slippery with probability at most $\frac{1}{100 |V(H)|^2}$.
    \end{claim}
    \begin{proof}
        Any $w$ such that $\text{Im}_{\pi_w}(V(H))$ is slippery can be specified as follows:
        \begin{enumerate}
            \item Choose an index $i$, and indices $j_1, \dots, j_r$.
            \item Choose $r$ vertices from $V(G)$ to serve as $\pi_w(v_{j_1}), \dots, \pi_w(v_{j_r})$.
            \item Choose the entries $w_\ell$ for $\ell \not\in \{i, j_1, \dots, j_r\}$.
            \item Choose the entry $w_i$, ensuring that $\pi_w(v_i)$ electrocutes $\pi_w(v_{j_1}), \dots, \pi_w(v_{j_r})$.
            \item Set $w_{j_1}, \dots, w_{j_r}$ to be the unique values such that $\pi_w(v_{j_1}), \dots, \pi_w(v_{j_r})$ correspond to the vertices chosen on step ii.
        \end{enumerate}
        By upper bounding the number of available choices at each step, we can obtain an upper bound on the number of $w$ such that $\text{Im}_{\pi_w}(V(H))$ is slippery.
        \begin{enumerate}
            \item There are at most $|V(H)|^{r+1}$ ways to choose the indices
            \item There are at most $n^k$ ways to choose $\pi_w(v_{j_1}), \dots, \pi_w(v_{j_r})$, where $k$ is the number of distinct elements appearing among $j_1, \dots, j_r$ (note that the $r$-tuple of indices may contain repeated elements).
            \item There are at most $(n^{1 - 1/1.51r})^{|V(H)| - k - 1}$ ways to choose $w_\ell$ for $\ell \not\in \{i, j_1, \dots, j_r\}$.
            \item By \cref{lem:electrocute}, so long as $n^{1-1/1.51r} \geq 200s |V(H)|^2$, there are at most $\left(200s |V(H)|^2\right)^s$ vertices that electrocute $\pi_w(v_{j_1}),\dots, \pi_w(v_{j_r})$.
        \end{enumerate}
        Overall, this means that the number of $w$ such that $\text{Im}_{\pi_w}(V(H))$ is slippery is at most 
        \begin{align*}
            &|V(H)|^{r+1} \cdot n^k \cdot (n^{1-1/1.51r})^{|V(H)| - k - 1} \cdot \left(200s |V(H)|^2\right)^s\\
            &= (n^{1-1/1.51r})^{\left(|V(H)| - k - 1 + \frac{k}{1 - 1/1.51r}\right)} \cdot \left(200s\right)^s|V(H)|^{2s + r + 1}\\
            &\leq (n^{1-1/1.51r})^{|V(H)|} \cdot n^{\left(\frac{1 - .51 r}{1.51r} \right)} \cdot \left(200s\right)^s|V(H)|^{2s + r + 1}.
        \end{align*}
        Since we know any $1$-degenerate $H$ has extremal number $O(n)$, we can assume $r > 1$, in which case $\left(\frac{1 - .51 r}{1.51r} \right) < 0$. So, for any constant values of $s$, $r$, and $|V(H)|$, for $n$ sufficiently large we have $n^{-\left(\frac{1 - .51 r}{1.51r} \right)} > 100 \left(200s\right)^s|V(H)|^{2s + r + 3}$, which means that the number of $w$ such that $\text{Im}_{\pi_w}(V(H))$ is slippery is at most $\frac{(n^{1-1/1.51r})^{|V(H)|}}{100 |V(H)|^2}$. Since the total number of $w$ is exactly $(n^{1-1/1.51r})^{|V(H)|}$, the slippery tuples represent less than a $\frac{1}{100 |V(H)|^2}$ fraction.
    \end{proof}

    We can now show that any particular non-edge of $H$ is very unlikely to be present in $\pi_w(H)$.
    
    \begin{claim}
        Let $v_i, v_j \in V(H)$ be any pair of vertices such that $(v_i, v_j) \not\in E(H)$. If $n$ is sufficiently large, and $w$ is chosen uniformly from $[n^{1 - 1/1.51r}]^{|V(H)|}$, then $\Pr_w[(\pi_w(v_i), \pi_w(v_j)) \in E(G)] \leq \frac{1}{50 |V(H)|^2}$.
    \end{claim}
    \begin{proof}
        We can assume without loss of generality that $i < j$. Let the tuple $u_1, \dots, u_r$ contain the neighbors of $v_i$ in $H$ that appear earlier in the degeneracy ordering (repeat a vertex in the tuple if there are fewer than $r$ such distinct vertices). If we fix random values for $w_1, \dots, w_{j-1}$, then this will in particular fix the embeddings $\pi_w(u_1), \dots, \pi_w(u_r)$ for all of those neighbors, as well as the embedding $\pi_w(v_i)$. \\
        
        Let $A = A(\pi_w(u_1),\dots, \pi_w(u_r))$; choosing a random value for $w_j$ will correspond to fixing $\pi_w(v_j)$ to be a uniform random element of $A$. By definition, unless $\pi(v_i)$ electrocutes $\pi_w(u_1), \dots, \pi_w(u_r)$, at most a $\frac{1}{100|V(H)|^2}$ fraction of the vertices of $A$ are adjacent to $\pi_w(v_i)$, so conditional on $\pi(v_i)$ not electrocuting $\pi_w(u_1), \dots, \pi_w(u_r)$ we have $\Pr_{w_j}[(\pi_w(v_i), \pi_w(v_j)) \in E(G)] \leq \frac{1}{100 |V(H)|^2}$. If $\pi(v_i)$ electrocutes $\pi_w(u_1), \dots, \pi_w(u_r)$, then $w$ is slippery --- so \cref{clm:not-slippery} ensures that this occurs with probability at most $\frac{1}{100|V(H)|^2}$. By union bound, this means $\Pr_w[(\pi_w(v_i), \pi_w(v_j)) \in E(G)] \leq \frac{1}{100 |V(H)|^2}+\frac{1}{100 |V(H)|^2}= \frac{1}{50 |V(H)|^2}$.
    \end{proof}

    Finally, we note that $\pi_w$ is injective with high probability: for any $i < j \leq |V(H)|$, we have $\Pr_{w}[\pi_w(v_i) = \pi_w(v_j)] \leq \frac{1}{n^{1 - 1/1.51}}$, because when $w_j$ is chosen, there are $n^{1-1/1.51}$ options, at most one of which corresponds to $\pi_w(v_i)$.\\

    Now, the probability that $\text{Im}_{\pi_w}(V(H))$ fails to be an induced copy of $H$ is by union bound at most 
    \begin{align*}
        &\left(\sum_{i <j} \Pr_{w}[\pi_w(v_i) = \pi_w(v_j)]\right) + \left(\sum_{i <j\colon (v_i, v_j) \not\in E(H)} \Pr[(\pi_w(v_i), \pi_w(v_j)) \in E(G)] \right) \\
        &\leq |V(H)|^2\cdot \frac{1}{n^{1 - 1/1.51}} + |V(H)|^2\cdot \frac{1}{50 |V(H)|^2}\\
        &\leq \frac{1}{25}
    \end{align*}
    for sufficiently large $n$. Since this probability is less than $1$, we know in particular that $G$ contains an induced copy of $H$.
\end{proof}

\begin{corollary}
    For any constant $\alpha$, if $\ex(n , H) \leq O(n^{\alpha})$, then $\ex(n, \{K_{s,s}, H\ind\}) \leq O\left(n^{\left(2 - \frac{(2- \alpha)^4}{320}\right)}\right)$.
\end{corollary}
\begin{proof}
    Let $r$ be the degeneracy of $H$. Since $\Omega(n^{2-2/r}) \leq \ex(n, H) \leq O(n^{\alpha})$, we must have $r \leq \frac{2}{2-\alpha}$. The result now follows from $\ex(n, \{K_{s,s}, H\ind\}) \leq O(n^{2 - 1/(20r^4)})$.
\end{proof}

\subsection{Towards better dependence on degeneracy: forbidding specific edges}\label{sec:strong-sauce}

The induced extremal number upper bound obtained in \cref{thm:degeneracy-controls-ind} is of the form $n^{2 - 1/\mathsf{poly}(r)}$, whereas for standard extremal numbers we know a bound of the form $n^{2-1/\Theta(r)}$. In the corresponding proof for standard extremal numbers, it suffices to apply the dependent random choice of \cref{lem:vanilla-drc} with $t = \Theta(1)$, guaranteeing that every $r$-tuple has a common neighborhood of size larger than some constant --- however, we took $t$ much larger in order to guarantee common neighborhoods of close to linear size. \\

The reason this was necessary was because our proof was counting ``out-of-order". We described a process of choosing an embedding that, when followed in degeneracy order, had exactly $n^{1 - 1/1.51r}$ choices at each step. However, in order to bound the number of slippery embeddings, we first had to fix the embeddings of the tuple that got electrocuted, and only then could count the number of choices for the vertex that electrocuted them. If the electrocuter appeared before the electrocutees in the degeneracy order, then this meant that we could not just embed in order, but instead had to fix the images of the electrocutees first, allowing them $n$ possibilities each as opposed to $n^{1-1/1.51r}$.  \\

One might wonder whether this technical issue can be overcome to show an upper bound of the form $n^{2-1/\Theta(r)}$. In this section, we make partial progress towards that goal, finding copies of the pattern subgraph which, while not necessarily induced, avoid particular subsets of the pattern graph's non-edges. Our first such result recovers bounds of the form $n^{2-1/\Theta(r)}$ when only a constant number of $H$'s non-edges must be preserved.

\begin{definition}
    For a graph $H$, and a subset $F \subseteq (V(H) \times V(H)) \setminus E(H)$ of ``forbidden'' edges, let $H\setminus (F\ind)$ denote the family of graphs $H'$ on $V(H)$ such that
    \begin{itemize}
        \item $(u,v) \in E(H) \implies (u,v) \in E(H')$, and
        \item $(u,v) \in F \implies (u,v) \not\in E(H')$.
    \end{itemize}
\end{definition}

\begin{proposition}
    For all $H$ of degeneracy $r$, and $F\subseteq (V(H) \times V(H)) \setminus E(H)$ with $|F| = f$, we have $\ex(n, \{K_{s,s}, H\setminus (F\ind)\}) \leq O(n^{2 - 1/(12f + 6r)})$.
\end{proposition}
\begin{proof}
    Let $V(F) \subseteq V(H)$ be the set of vertices with an endpoint in $F$, noting that $|V(F)| \leq 2f$. Consider an $n$-vertex $K_{s,s}$-free graph $G$ on $n^{2 - 1/(12f + 6r)}$ vertices. Applying \cref{lem:vanilla-drc} with $t = 1.81$, we find vertex subsets $U_1, U_2 \subseteq V(G)$ such that any $(r+2f)$-tuple of vertices in one subset has at least $n^{.001}$ common neighbors in the other. The subgraph of $H$ induced by $V(F)$ has at most $2f$ vertices, and thus maximum degree at most $2f$ --- so, by the results of~\cite{main-paper} on induced extremal numbers of graphs with bounded maximum degree, we can find an induced copy of that subgraph. Fix that subgraph as the embedding of $V(F)$, order the remaining vertices of $H$ in degeneracy order, and them embed one-at-a-time. Each vertex to be embedded is neighbors with at most $r + 2f$ already-embedded vertices, so there are at least $n^{.001}$ candidates. As long as $n$ is large enough that $n^{.001} > |V(H)|$, this ensures that there is always an option that has not already been used, and so a copy of $H\setminus(F\ind)$ can be found.
\end{proof}

We also observe that it is possible to forbid all edges between vertices that are close to each other in degeneracy order.

\begin{theorem}\label{thm:degenerate-no-close-edges}
    Let $H$ be an $r$-degenerate bipartite graph, and $v_1, \dots, v_{|V(H)|}$ be an ordering of the vertices of $H$ such that, for all indices $i$, there are at most $r$ indices $j <i$ with $(v_i, v_j) \in E(H)$. Then, for any $q \in \N$, we have $\ex(n, \{K_{s,s}, H\setminus (F\ind)\}) \leq O(n^{2-1/(2000q^2r)})$, where $F = \{(v_i, v_j)\colon (v_i, v_j) \not\in E(H) \text{ and } |i - j| \leq  q\}$.
\end{theorem}

To prove \cref{thm:degenerate-no-close-edges}, we will need to re-do the analysis of dependent random choice. We are interested in obtaining something complementary to \cref{lem:electrocute}: we want our sets to be such that, for any $r$-tuple of vertices in one set, although few vertices are neighbors with a \emph{very large} fraction of the tuple's common neighborhood, \emph{every} vertex is neighbors with at least a \emph{somewhat large} fraction of the tuple's common neighborhood. 

\begin{lemma}\label{lem:eel-drc}
    For any $r, t \geq 2$, any sufficiently large $n$, and any $n$-vertex graph $G$ with at least $n^{2 - 1/(9t^2r)}$ edges, there exist nonempty subsets $U_1, U_2 \subseteq V(G)$ such that both of the following conditions hold.
    \begin{itemize}
        \item Any $r$-tuple of vertices $(u_1, \dots, u_r) \in U_i^r$ has a large common neighborhood --- that is, $|\bigcap_{j} N(u_j) \cap U_{3-i}| \geq n^{1/10}$. 
        \item For any $r$-tuple of vertices $(u_1, \dots, u_r) \in U_i^r$, and any other vertex $v \in U_i$, a sizeable fraction of the common neighborhood of $(u_1, \dots, u_r)$ is neighbors with $v$ --- that is, $\frac{|\bigcap_i N(u_i) \cap U_{3-i} \cap N(v)|}{|\bigcap_i N(u_i) \cap U_{3-i}|} \geq n^{-1/t}$.
    \end{itemize}
\end{lemma}
\begin{proof}
    The proof is essentially the same as that of \cref{lem:vanilla-drc}, but we include it in full for completeness. First, partition $V(G)$ into two parts $L$ and $R$ such that at least half of the edges cross the partition. Then, randomly choose $q_L = 3t^2r$ vertices $\ell_1, \dots, \ell_{q_L} \in L$ with replacement, and consider their common neighborhood $A = \bigcap_i N(\ell_i) \bigcap R$. For $p = r + tr + 2t$, we will define random variables $X,Y,Z$, where we let $X = |A|$, let $Y$ be the number of tuples $(u_1, \dots, u_p) \in A^p$ such that $|\bigcap_i N(u_i) \cap L| < n^{1/10}$, and let $Z$ be the number of tuples $(u_1, \dots, u_p, v) \in A^{p+1}$ such that $\frac{|\bigcap_i N(u_i) \cap L \cap N(v)|}{|\bigcap_i N(u_i) \cap L|} < n^{-1/t}$. To bound $\mathbb{E}[X]$, we note that $\sum_{v\in R}|N(v) \cap L| \ge \frac{n^{2-1/(9t^2r)}}{2}$, and use convexity:
    \[\mathbb{E}[X] = \sum_{v \in R} \left(\frac{|N(v) \cap L|}{|L|} \right)^{q_L} \geq \frac{1}{n^{q_L}} \sum |N(v) \cap L|^{q_L} \geq \frac{1}{n^{q_L}} \cdot n \cdot \left(\frac{n^{2-1/(9t^2r)}}{2n}\right)^{q_L} \geq \frac{n^{2/3}}{2^{q_L}}.\]
    To calculate $\mathbb{E}[Y]$, we note that any fixed such tuple will only be chosen if all $q_L$ chosen vertices lie in the common neighborhood, so
    \[\mathbb{E}[Y] = \sum_{\substack{(u_1, \dots, u_r)\in R^r \\ |\bigcap_i N(u_i) \cap L| \leq n^{1/10}}} \Pr[u_1, \dots, u_p \in A ] \leq n^p \cdot \left(\frac{n^{1/10}}{n}\right)^{q_L} < 1.\]
    
    To bound $\mathbb{E}[Z]$, we note that $\Pr_{\ell_1, \dots, \ell_{q_L}}[u_1, \dots, u_p \in A \text{ and } v \in A] \leq \Pr_{\ell_1, \dots, \ell_{q_L}}[v \in A\ | \ u_1,\dots,u_p \in A] = \left(\frac{|\bigcap_i N(u_i) \cap L \cap N(v)|}{|\bigcap_i N(u_i) \cap L|}\right)^{q_L}$. So, bounding the number of tuples $(u_1, \dots, u_p, v) \in R^{p+1}$ such that $\frac{|\bigcap_i N(u_i) \cap L \cap N(v)|}{|\bigcap_i N(u_i) \cap L|} < n^{-1/t}$ by $n^{p+1}$, we obtain

    \[\mathbb{E}[Z] \leq n^{p + 1} \cdot \left(n^{-1/t}\right)^{q_L} < 1.\]
    
    Since $\mathbb{E}[X - Y - Z] \ge \frac{n^{2/3}}{2^{q_L}} - 2 > n^{2/3 - .0001}$, there exists some choice of $\ell_1, \dots, \ell_{q_L}$ such that $X - Y - Z> n^{2/3 - .0001}$. Fix $A$ according to this choice of $\ell$, and let $U_1$ be obtained from $A$ by removing one vertex from each tuple $(u_1, \dots, u_r) \in A^r$ such that $|\bigcap_i N(u_i) \cap L| < n^{1/2}$, and one vertex from each tuple $(u_1, \dots, u_r, v) \in A^{r+1}$ such that $\frac{|\bigcap_i N(u_i) \cap L \cap N(v)|}{|\bigcap_i N(u_i) \cap L|} < n^{-1/(2r)}$. We are guaranteed that $|U_1| \geq n^{2/3-.0001}$.\\

    Now, choose $q_R = t(r+2)$ vertices $r_1, \dots, r_{q_R} \in U_1$, and consider their common neighborhood $B = \bigcap_i N(r_i) \cap L$. Let $Y'$ be the number of tuples $(u_1, \dots, u_r) \in B^r$ such that $|\bigcap_i N(u_i) \cap U_1| < n^{1/10}$, and let $Z'$ be the number of tuples $(u_1, \dots, u_r, v) \in B^{r+1}$ such that $\frac{|\bigcap_i N(u_i) \cap U_1 \cap N(v)|}{|\bigcap_i N(u_i) \cap U_1|} < n^{-1/t}$. By the same logic as above, we have
    
    \[\mathbb{E}[Y'] \leq n^r \cdot \left( \frac{n^{1/10}}{n^{2/3-.0001}} \right)^{q_R} < 1/2,\]
    \[\mathbb{E}[Z'] \leq  n^{r + 1} \cdot \left(n^{-1/t}\right)^{q_R} < 1/2.\]
    
    So, by a union bound, there exists a choice of the $r_1, \dots, r_{q_R}$ such that $Y' = Z' = 0$. Set $U_2$ to be the corresponding set of $B$.\\

    We claim $U_1, U_2$ satisfy the conditions of the statement. The fact that $Y' =Z'=0$ immediately implies the conditions for vertices in $U_2$. Then, since $r + q_R = p$, for any $(u_1, \dots, u_r) \in U_1$ we have

    \[\left|\bigcap_i u_i \cap U_2\right| = \left|\bigcap_i u_i \cap \left(\bigcap_i N(r_i) \cap L\right)\right| \geq n^{1/10}, \]
    
    and for any $(u_1, \dots, u_r, v) \in U_1^{r+1}$, we have

    \[\frac{|\bigcap_i N(u_i) \cap U_{2} \cap N(v)|}{|\bigcap_i N(u_i) \cap U_{2}|} = \frac{|\bigcap_i N(u_i) \cap \left(\bigcap_i N(r_i)\cap L\right) \cap N(v)|}{|\bigcap_i N(u_i) \cap \left(\bigcap_i N(r_i) \cap L\right)|} \geq n^{-1/t}.\]
\end{proof}

We can now run a proof strategy similar to that of \cref{thm:degeneracy-controls-ind}, since \cref{lem:eel-drc} will let us show that the number of available embedding options does not grow too much when we embed \emph{only a little bit} out-of-order.

\begin{proof}[Proof of \cref{thm:degenerate-no-close-edges}]

Let $G$ be an $n$-vertex $K_{s,s}$-free graph with at least $n^{2 - 1/(2000q^2r)}$ edges. We can apply \cref{lem:eel-drc} with $t = 11q$ to find two parts $U_1, U_2 \subseteq G$ such that any $r$-tuple in one part has at least $n^{1/10}$ common neighbors in the other part, and any additional vertex is neighbors with at least an $n^{-1/11t}$ fraction of that common neighborhood.\\

As in the proof of \cref{thm:degeneracy-controls-ind}, we now define a distribution on homomorphisms $H \to G$, and show that a random homomorphism from this distribution is likely to correspond to a copy of $H$ as a subgraph without any of the forbidden edges. Once again, we will do so by embedding vertices in degeneracy order, choosing an embedding at each step uniformly from a prefix of the list of available candidates. However, in this case instead of defining all of these prefixes in terms of fixed, arbitrary orderings of the left vertex part $U_1$ and the right vertex part $U_2$, it will be useful to choose new, random orderings for each step of the embedding. For any tuple of numbers $w = (w_1, \dots, w_{|V(H)|}) \in [n^{1/10}]^{|V(H)|}$, and any tuple of permutations $\sigma = \big((\sigma_1^{(1)}, \sigma_1^{(2)}), \dots, (\sigma_{|V(H)|}^{(1)}, \sigma_{|V(H)|}^{(2)})\big) \in \big( S_{U_1} \times S_{U_2}\big)^{|V(H)|}$, we will define $\pi_w^{(\sigma)}: V(H) \to V(G)$.
If $v_i$ belongs to the left part of $H$, then $\pi_w^{(\sigma)}(v_i)$ is the $w_i$th vertex of $\bigcap_{j \leq i\colon (v_i, v_j) \in E(H)} N(v_j) \cap U_1$ to appear in the ordering $\sigma_i^{(1)}$. Likewise, if $v_i$ belongs to the right part of $H$, then $\pi_w^{(\sigma)}(v_i)$ is the $w_i$th vertex of $\bigcap_{j \leq i\colon (v_i, v_j) \in E(H)} N(v_j) \cap U_1$ to appear in the ordering $\sigma_i^{(2)}$.\\

We can make a slightly simpler definition of electrocution here, saying that $v \in U_i$ \defn{electrocutes} $u_1, \dots, u_r \in U_i$ if $|\bigcap_i N(u_i) \cap U_j \cap N(v)| \geq \frac{1}{100|V(H)|^2} |\bigcap_i N(u_i) \cap U_j|$. We will also define a version of slipperiness that requires indices to be close to each other: say that a tuple of vertices $Y = (y_1, \dots, y_{|Y|}) \in G^{|Y|}$ is \defn{slippery} if there exist $
(j_1, \dots, j_r) \in [|Y|]^r$ and $j^* \in Y \setminus \{j_1, \dots, j_r\}$ such that $y_{j^*}$ electrocutes $(y_{j_1}, \dots, y_{j_r})$, and also $\max_i (j_i) - j^* \leq q$.

\begin{claim}\label{clm:no-local-slippery}
    If $n$ is sufficiently large, $w$ is chosen uniformly from $[n^{1/10}]^{|V(H)|}$, and $\sigma$ is chosen uniformly from $\big( S_{U_1} \times S_{U_2}\big)^{|V(H)|}$, then the image $\text{Im}_{\pi_w^{(\sigma)}}(V(H))$ is slippery with probability at most $\frac{1}{100|V(H)|^2}$.
\end{claim}
\begin{proof}
    By a union bound, it suffices to show for every particular $j_1 \leq \dots \leq j_r$ and every $j^*$ with $j_r - j^* \leq q$ that the probability of $\pi_{w}^{(\sigma)}(v_{j^*})$ electrocuting $\left(\pi_{w}^{(\sigma)}(v_{j_1}), \dots, \pi_{w}^{(\sigma)}(v_{j_r}) \right)$ is at most $\frac{1}{100|V(H)|^{r+3}}$. Note that if $j_r < j^*$ this is easy:
    conditioning on any values of $\sigma$ and $w_1, \dots, w_{j^*-1}, w_{j^*+1}, \dots, w_{|V(H)|}$, by \cref{lem:electrocute} at most $(200s|V(H)|^2)^{2s}$ of the $n^{1/10}$ possible choices for $w_{j^*}$ result in $\pi_w^{(\sigma)}(v_{j^*})$ electrocuting $\left(\pi_{w}^{(\sigma)}(v_{j_1}), \dots, \pi_{w}^{(\sigma)}(v_{j_r}) \right)$. We would like to say something similar when $j_r$ is larger (but not much larger) than $j^*$.\\
    
    Fix some such $j_1, \dots, j_r, j^*$, and also fix any values for $w_1, \dots, w_{j^*-1}$ and $\sigma_1, \dots, \sigma_{j^*-1}$ --- we claim that the electrocution probability is low conditioned on any such choices. Note that the probability of electrocution now depends only on the values of $w_{j^*}, \dots, w_{j_r}$ and $\sigma_{j^*}, \dots, \sigma_{j_r}$, since indices later than $j_r$ have no effect on the embeddings of earlier vertices. In order to count the number of embeddings that are slippery at these indices, we will divide into two types, whose counts we will bound separately. For any $i$ such that $j^* < i \leq j_r$, letting $b \in \{1,2\}$ be the part of $H$ to which $v_{i}$ belongs, we let $A_{i} = \bigcap_{\ell < i \colon (v_\ell, v_{i}) \in E(H)) \text{ and } \ell \neq j^*} N(\pi_w^{(\sigma)}(v_\ell)) \cap U_b$ be the set of candidates for $\pi_w^{(\sigma)}(v_{i})$ at the time of embedding when one ignores the potential requirement to be neighbors with $\pi_w^{(\sigma)}(v_{j^*})$. 
    
    \paragraph{Case 1:} For all $i$ such that $j^* < i \leq j_r$, at least $n^{1/10}$ vertices of $N(\pi_w^{(\sigma)}(v_{j^*}))$ appear among the first $100 n^{1/10 + 1/11q}$ vertices of $A_{i}$ in the ordering $\sigma_{i}^{(b)}$.\\

    In order to bound the probability of electrocution in this case, we can condition on any arbitrary value of $\sigma$. Now, to bound the number possible choices of $w_{j^*}, \dots, w_{j_r}$, we can do the following: 
    \begin{enumerate}
        \item For each $i$ such that $j^* < i \leq j_r$ in order, choose a value for $\pi_w^{(\sigma)}(v_i)$ among the first $100 n^{1/10 + 1/11q}$ vertices of $A_i$ in ordering $\sigma_{i}^{(b)}$.
        \item Choose a value for $\pi_w^{(\sigma)}(v_{j^*})$ among the vertices that electrocute $\pi_w^{(\sigma)}(v_{j_1}), \dots, \pi_w^{(\sigma)}(v_{j_r})$.
    \end{enumerate}

    Note that, since $\sigma$ is fixed, fixing the embeddings $\pi_w^{(\sigma)}(v_{j^*}), \dots, \pi_w^{(\sigma)}(v_{j_r})$ will uniquely determine the values $w_{j^*}, \dots, w_{j_r}$. In any valid embedding, for all $i$ with $j^* < i \leq j_r$, we will either choose $v_i$ from among the first $n^{1/10}$ vertices of $A_i$ (if $(v_{j^*}, v_i) \not\in E(H)$), or among the first $n^{1/10}$ vertices of $A_i \cap N(\pi_w^{(\sigma)}(v_{j^*})$ (if $(v_{j^*}, v_i) \in E(H)$). So, the number of ways to perform the above process is indeed an upper bound on the number of choices of $w_{j^*}, \dots, w_{j_r}$ that lead to electrocution in this case. Since, by \cref{lem:electrocute}, there are at most $(200s|V(H)|^2)^s$ vertices that electrocute any fixed tuple $\pi_w^{(\sigma)}(v_{j_1}), \dots, \pi_w^{(\sigma)}(v_{j_r})$, the number of ways to perform the above process is at most 
    \begin{align*}
        \left( 100 n^{1/10 + 1/11q}\right)^{j_r - j^*} \cdot (200s|V(H)|^2)^s &\leq n^{(j_r-j^*)/10 + (j_r - j^*)/11q} \cdot (200s|V(H)|^2)^{2s + q} \\
        &\leq \left(n^{1/10}\right)^{j_r + 1 - j^*} \cdot \frac{1}{200 |V(H)|^{r+3}}
    \end{align*}
    for sufficiently large $n$. So the probability of both belonging to case 1 and having $\pi_w^{(\sigma)}(v_{j^*})$ electrocute $\pi_w^{(\sigma)}(v_{j_1}), \dots, \pi_w^{(\sigma)}(v_{j_r})$ is at most $\frac{1}{200|V(H)|^{r+3}}$.

    \paragraph{Case 2:} For some $i^*$ with $j^* < i^* \leq j_r$, fewer than $n^{1/10}$ vertices of $N(\pi_w^{(\sigma)}(v_{j^*}))$ appear among the first $100 n^{1/10 + 1/11q}$ vertices of $A_{i^*}$ in the ordering $\sigma_{i}^{(b)}$.\\

    We claim that, electrocution aside, case 2 is very unlikely. Fix any values for $w_{j^*}, \dots, w_{i^*- 1}$ and $\sigma_{j^*}, \dots, \sigma_{i^* - 1}$. By our dependent random choice, we have guaranteed that $|A_{i^*} \cap N(\pi_w^{(\sigma)}(v_{j^*}))| \geq n^{-1/11q} \cdot |A_{i^*}|$. So, in expectation over $\sigma_{i^*}$ there will be at least $100 n^{1/10}$ vertices of $N(\pi_w^{(\sigma)}(v_{j^*}))$ among the first $100 n^{1/10 + 1/11q}$ vertices of $A_{i^*}$. By a Chernoff bound, the probability of lying substantially below this expectation is extremely small. That is, the probability of having fewer than $n^{1/10}$ vertices of $N(\pi_w^{(\sigma)}(v_{j^*}))$ among the first $100n^{1/10 + 1/11q}$ vertices of $A_{i^*}$ is at most the probability of fewer than $n^{1/10}$ successes in $100n^{1/10 + 1/11q}$ flips of an $n^{-1/11q}$-biased coin. This occurs with probability at most $2^{-2(99n^{1/10})^2/(100n^{1/10+1/11q})} \leq o(2^{-n^{0.0001}})$ \cite{pm}. So, union bounding over all possible $i^*$, for $n$ sufficiently large the probability of lying in case 2 is at most $\frac{1}{200|V(H)|^{r+3}}$.\\

    Since cases 1 and 2 are exhaustive, we have upper bounded the probability of $\pi_{w}^{(\sigma)}(v_{j^*})$ electrocuting $\left(\pi_{w}^{(\sigma)}(v_{j_1}), \dots, \pi_{w}^{(\sigma)}(v_{j_r}) \right)$ by $\frac{1}{200|V(H)|^{r+3}} + \frac{1}{200|V(H)|^{r+3}} = \frac{1}{100|V(H)|^{r+3}}$.
\end{proof}

    Once again, we can now say that forbidden edges are unlikely unless the embedding is slippery.

    \begin{claim}
        Let $v_i, v_j \in V(H)$ be any pair of vertices such that $(v_i, v_j) \not\in E(H)$, and $| i - j | \leq q$. If $n$ is sufficiently large, $w$ is chosen uniformly from $[n^{1/10}]^{|V(H)|}$, and $\sigma$ is chosen uniformly from $\big( S_{U_1} \times S_{U_2}\big)^{|V(H)|}$, then $\Pr_{w,\sigma}[(\pi_w^{(\sigma)}(v_i), \pi_{w}^{(\sigma)}(v_j)) \in E(G)] \leq \frac{1}{50 |V(H)|^2}$.
    \end{claim}
    \begin{proof}
        Assume without loss of generality that $i<j$. Fix random values for $w_1, \dots, w_{j-1}$ and $\sigma_1, \dots, \sigma_{j-1}$. Now, choosing uniform random values for $\sigma_j$ and $w_j$ will cause $\pi_w^{(\sigma)}(v_j)$ to be chosen as a uniform random vertex of $\bigcap_{\ell < j\colon (v_\ell, v_j) \in E(H)} N(v_\ell) \cap U_b$, where $b \in \{1,2\}$ is the index of the part of $H$ to which $v_j$ belongs. If the overall embedding is not slippery, then at most a $\frac{1}{100 |V(H)|^2}$ fraction of this set belongs to $N(\pi_w^{(\sigma)}(v_i))$. So, by \cref{clm:no-local-slippery} we have $\Pr_{w,\sigma}[(\pi_w^{(\sigma)}(v_i), \pi_{w}^{(\sigma)}(v_j)) \in E(G)] \leq \frac{1}{100 |V(H)|^2} + \frac{1}{100 |V(H)|^2} = \frac{1}{50 |V(H)|^2}$.
    \end{proof}

    Once again, the probability of $\pi_w^{(\sigma)}$ failing to be injective goes to $0$ in $n$, so union bounding this along with the probability of any forbidden edge existing gives overall probability strictly less than $1$ for large $n$. Thus, there exists some copy of $H$ in $G$ as a subgraph that avoids inducing any edge of $F$.
\end{proof}

\section{Possible connected counterexamples}\label{sec:real-counterexs}

In \cref{sec:degeneracy}, we have shown an upper bound on $\ex(n, \{K_{s,s}, H\ind\})$ in terms of $\ex(n, H)$, which it would be interesting to improve by strengthening the quantitative bounds of \cref{thm:degeneracy-controls-ind}. However, even the best possible control by degeneracy one could hope for would not suffice to demonstrate \cref{conj:connected}, because it is known the degeneracy does not completely determine standard extremal numbers. One is left with the question: is \cref{conj:connected} likely to be true? In this section, for the benefit of the unbelievers, we briefly discuss a potential source of counterexamples.\\

The simplest setting to consider would be where $s = 2$. Recall that $\ex(n, K_{2,2}) = \Theta(n^{3/2})$, where the lower bound is attained by the incidence graph of all points and lines over a finite projective plane $\PG(2,q)$~\cite{brown1966c4}. These projective plane incidence graphs are highly structured; in addition to avoiding $K_{2,2}$, they may avoid many other interesting structures. A possible approach to disproving \cref{conj:connected} would be to find some subgraph with extremal number $o(n^{3/2})$ which is nonetheless avoided in \emph{induced} form by some such family of incidence graphs. \\

Indeed, one can find examples of pattern graphs $H$ such that the point-line incidence graph of $\PG(2,q)$ must always contains many copies of $H$, but where none of the copies are induced. Perhaps the simplest example would be the the Heawood graph with one edge deleted. The \defn{Heawood graph}, which we will denote $\Hea$, is the incidence graph of the \defn{Fano plane} (i.e., all points and lines over $\PG(2,2)$) --- let $\Hea^- = \Hea \setminus e$ be the Heawood graph with a single edge deleted (note that $\Hea$ is edge-transitive, so we need not specify which edge is removed).\\

\begin{figure}[h]
    \centering
    \includegraphics[width=0.3\linewidth]{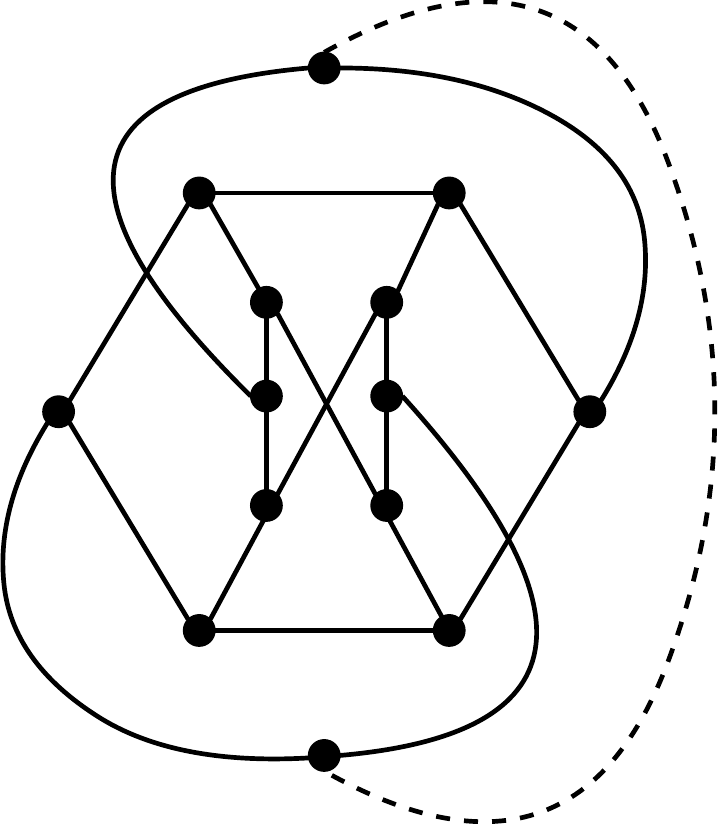}
    \caption{$\Hea^-$, the incidence graph of the Fano plane with a single edge deleted. Deleted edge shown dashed.}
    \label{fig:heawood-minus}
\end{figure}

\begin{proposition}
    $\ex(n, \{K_{2,2}, \Hea^-\ind\}) = \Theta(n^{3/2})$.
\end{proposition}
\begin{proof}
    A \defn{complete quadrangle} is a set of four points, no three of which are colinear, and the $6$ lines between each pair --- the \defn{diagonals} of a complete quadrangle are the three additional intersection points of those lines. For any finite field $\F_{q}$, we know the diagonals of a complete quadrangle in $\PG(2,q)$ will be colinear if and only if $q$ is a power of $2$ ~\cite{projectivebook}.\\ 
    
\begin{figure}[h]
    \centering
    \includegraphics[width=0.18\linewidth]{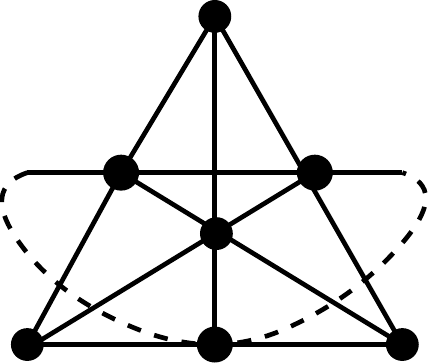}
    \caption{A complete quadrangle, with a line between two of the three diagonal points. The dashed extension of that line indicates that it passes through the third diagonal if and only if the underlying field has characteristic $2$ (in which case the configuration is isomorphic to the Fano plane).}
    \label{fig:fano}
\end{figure}

    In particular, this means that $\Hea^-$ cannot appear as an induced subgraph of the point-line incidence graph of $\PG(2,2^a)$ for any $a\in \N$: given $7$ points and $6$ lines corresponding to a complete quadrangle, a line incident to two of the diagonals must also be incident to the third, thus inducing $\Hea$. So incidence graphs of $\PG(2, 2^a)$ give a family of $n$-vertex graphs with $\Theta(n^{3/2})$ edges avoiding both $K_{2,2}$ as a subgraph and $\Hea^-$ as an induced subgraph. 
\end{proof}

If we knew that $\ex(n, \Hea^-) < o(n^{3/2})$, then this would therefore be a counterexample to \cref{conj:connected}. It seems perhaps plausible that $\Hea^-$ could have a small extremal number: it is a 2-degenerate graph of girth 6, and appears as a subgraph (although not necessarily induced) of every complete point-line incidence graph over a finite projective plane. The strongest lower bound we know on its extremal number is $\ex(n, \Hea^-) \geq \Omega(n^{7/5})$, obtained by considering a random host graph. It would be quite interesting to determine tight bounds on the extremal number of $\Hea^-$, and thus determine whether or not it represents a counterexample to \cref{conj:connected}. However this may be a difficult task: there are remarkably few graphs for which the true Tur\'an exponent is known, with the techniques involved typically quite specialized to the particular graph in question~\cite{furedi-history}.\\

\begin{figure}[h]
    \centering
    \includegraphics[width=0.7\linewidth]{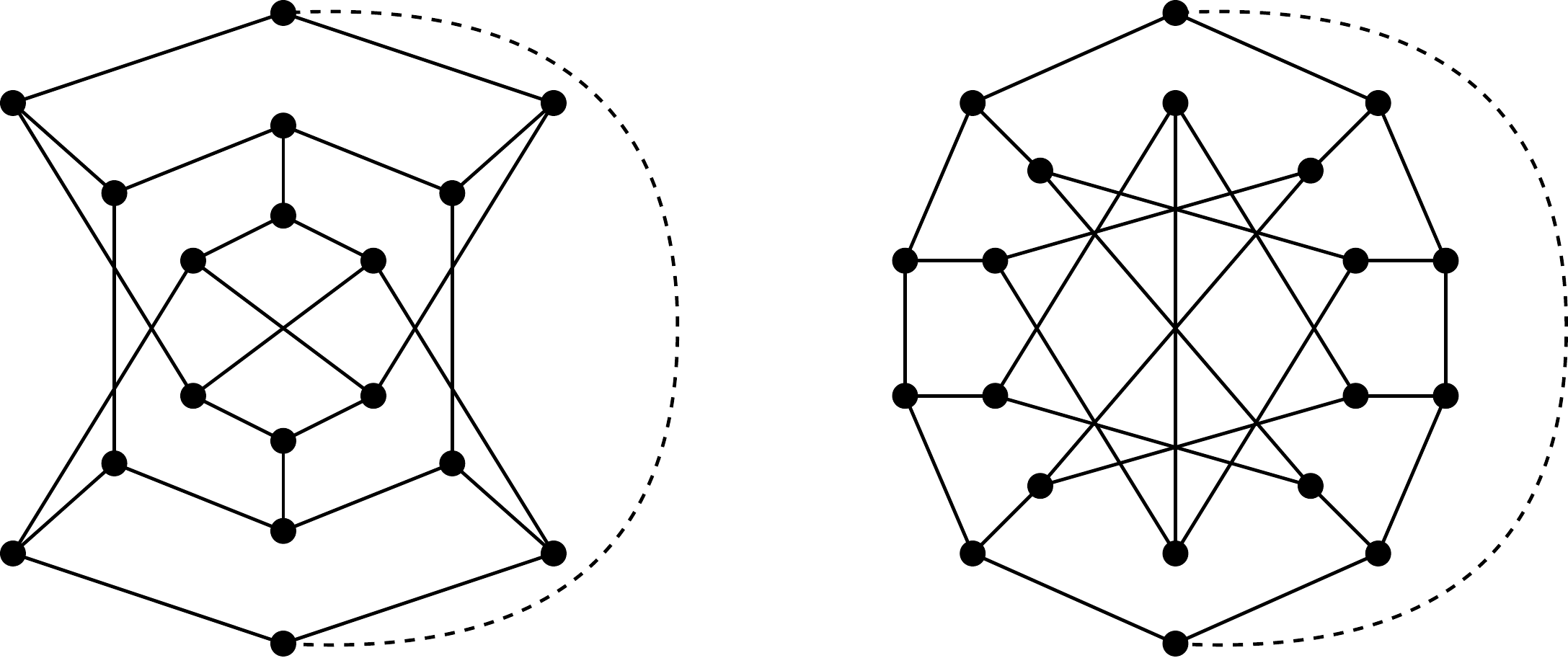}
    \caption{Two graphs $H$ where the bound $\ex(n, \{K_{2,2}, H\ind\}) = \Theta(n^{3/2})$ can be obtained from incidence geometry theorems: the Pappus graph (left), and the Desargues graph (right), each with a single deleted edge (denoted by a dashed line). Pappus's theorem and Desargues's theorem, respectively, ensure that neither can appear as an induced subgraph of $\PG(2, q)$ for any $q$, since the dashed edge will always be present.}
    \label{fig:pappus-and-desargues}
\end{figure}

The Heawood graph is, of course, far from the only case of such a structure: there are many more complicated theorems demonstrating that, in some particular incidence configuration, three points must be colinear \cite{richter1995incidences, fomin2023incidences}. Such theorems will allow us to find subgraphs which appear in all projective plane incidence graphs, but can be avoided in induced form (see \cref{fig:pappus-and-desargues} for two additional examples). \cref{conj:connected} would hold that all such graphs have extremal number $\Omega(n^{3/2})$; evidence for or against that prediction could give intuition as to whether \cref{conj:connected} is likely to be true. As a point against the prediction, we note that Conlon has conjectured that every $2$-degenerate $C_4$-free bipartite graph has extremal number $O(n^{3/2 - \delta})$ for some $\delta$ (this conjecture is cited as personal communication in \cite{shapira2025new}) --- as all three graphs mentioned in this section fit those criteria, at least one of Conlon's conjecture and \cref{conj:connected} must be false.

\section{Acknowledgements}

The results in this paper were produced as a part of the Duluth REU program. We thank the organizers, advisors, visitors and participants of the program for offering mathematical advice and providing a stimulating research environment --- we are especially grateful to Maya Sankar and Carl Schildkraut for a number of helpful discussions. We also thank Joe Gallian and Colin Defant for running the program, and Jane Street Capital, Ray Sidney, Eric Wepsic, and NSF grant no. DMS-2409861 for providing funding. We thank Noah Kravitz, Carl Schildkraut, and Aleksa Milojevi\'c for giving feedback on earlier versions of this manuscript.

\printbibliography
\end{document}